\documentclass[11pt,oneside,russian,english,reqno]{amsart}
\usepackage[T1]{fontenc}
\usepackage[koi8-r,latin9]{inputenc}
\usepackage{babel}

\usepackage{url}
\usepackage{amsthm}
\usepackage{amstext}
\usepackage{amssymb}
\usepackage{esint}
\usepackage[unicode=true, pdfusetitle,
 bookmarks=true,bookmarksnumbered=false,bookmarksopen=false,
 breaklinks=true,pdfborder={0 0 0},backref=false,colorlinks=false]
 {hyperref}

\makeatletter

\newcommand{\noun}[1]{\textsc{#1}}

\AtBeginDocument{\DeclareFontEncoding{T2A}{}{}}

\providecommand{\tabularnewline}{\\}

\theoremstyle{plain}
\newtheorem{thm}{Theorem}
  \theoremstyle{plain}
  \newtheorem*{cor*}{Corollary}
 \theoremstyle{definition}
 \newtheorem*{defn*}{Definition}
  \theoremstyle{plain}
  \newtheorem{lem}{Lemma}
  \theoremstyle{remark}
  \newtheorem*{rem*}{Remark}

\usepackage{fourier}
\usepackage{graphics}
\usepackage{color}

\renewcommand{\theenumi}{(\roman{enumi})}
\renewcommand{\labelenumi}{\theenumi}

\DeclareMathOperator{\Cov}{Cov}
\DeclareMathOperator{\Haus}{Haus}
\DeclareMathOperator{\Mink}{Mink}
\DeclareMathOperator{\F}{F}
\DeclareMathOperator{\lp}{l}
\DeclareMathOperator{\dist}{dist}
\DeclareMathOperator{\diam}{diam}
\theoremstyle{plain}

\newtheorem{NIH}{Theorem}

\theoremstyle{remark}

\def\righteqn#1{\llap{$\displaystyle #1$}}

\makeatother

\begin{document}

\title[Singular distributions and dimension]{Singular distributions, dimension of support, and symmetry of Fourier
transform}

\author{Gady Kozma and Alexander Olevski\u\i}

\address{GK: Department of Mathematics, The Weizmann Institute of Science,
Rehovot POB 76100, Israel.\\
AO: School of Mathematics, Tel Aviv University, Tel Aviv 69978,
Israel.}

\email{gady.kozma@weizmann.ac.il, olevskii@post.tau.ac.il}


\begin{abstract}
We study the {}``Fourier symmetry'' of measures and distributions
on the circle, in relation with the size of their supports. The main
results of this paper are:
\begin{enumerate}
\item A one-side extension of Frostman's theorem, which connects the rate
of decay of Fourier transform of a distribution with the Hausdorff
dimension of the support;{\small \par}
\item A construction of compacts of {}``critical'' size, which support
distributions (even pseudo-functions) with anti-analytic part belonging
to $l^{2}$.{\small \par}
\end{enumerate}
We also give examples of non-symmetry which may occur for measures
with {}``small'' support. A number of open questions are stated.{\small \par}

\medskip{}

\noindent \textsc{Résumé.} On étudie la {}``symétrie de Fourier''
des mesures et des distributions sur le cercle en rapport avec la
dimension de leurs supports. Les résultats essentiels du présent travail
sont les suivants:
\begin{enumerate}
\item L'extension unilatérale du théorème de Frostman qui met en rapport
la vitesse de décroissance de la transformation de Fourier d'une distribution
et la dimension de Hausdorf de son support.{\small \par}
\item La construction des compacts d'une taille {}``critique'' qui peut
supporter des distributions (voire des pseudo-fonctions) avec une
partie anti-ana\-lytique appartenant à $l^{2}$.{\small \par}
\end{enumerate}
On donne également quelques exemples de l'asymétrie qui peut se produire
pour des mesures à {}``petit'' support. Plusieurs questions ouvertes
sont formulées.
\end{abstract}
\maketitle

\section*{I}

Let \emph{$K$} be a compact subset of the circle group $\mathbb{T}$.
Frostman's theorem allows to characterize the Hausdorff dimension
of $K$ by examining (non-trivial) measures supported on $K$. The
most common version states that \[
\dim K=\sup\left\{ \alpha:\exists\mu\mbox{ supported on }K\mbox{ with }\int\!\!\!\int\frac{d\mu(x)\, d\mu(y)}{|x-y|^{\alpha}}<\infty\right\} .\]
It is not difficult to translate this theorem to the language of Fourier
coefficients. One then gets\begin{equation}
\dim K=\sup\left\{ \alpha:\exists\mu\mbox{ supported on }K\mbox{ with }\sum_{n=-\infty}^{\infty}\frac{|\widehat{\mu}(n)|^{2}}{|n|^{1-\alpha}+1}<\infty\right\} .\label{eq:twosidedfull}\end{equation}
Beurling showed that one may replace measures with arbitrary (Schwartz)
distributions. The following version of Frostman's theorem thus holds
(\cite{B49} or \cite[p.~40]{KS94}),

\begin{NIH}Let $0<\alpha\le1$. Then
\begin{enumerate}
\item \label{enu:S->dim}If there is a (non-trivial) distribution $S$ supported
on $K$ such that\begin{equation}
\sum_{n=-\infty}^{\infty}\frac{|\widehat{S}(n)|^{2}}{|n|^{1-\alpha}+1}<\infty,\label{eq:twosided}\end{equation}
then $\dim K\geq\alpha$.
\item If $\dim K>\alpha$ then there is probability measure $S$, satisfying
(\ref{eq:twosided}).
\end{enumerate}
\end{NIH}

We will show that a one-side estimate of $\widehat{S}$ is already
sufficient in \ref{enu:S->dim}. 
\begin{thm}
\label{thm:main}Suppose that there is a non-trivial Schwartz distribution
$S$ supported on $K$ with\begin{equation}
\sum_{n=-\infty}^{-1}\frac{|\widehat{S}(n)|^{2}}{|n|^{1-\alpha}}<\infty.\label{eq:one-sided}\end{equation}
Then $\dim K\ge\alpha$.
\end{thm}
We denote by $\mathcal{D}^{*}$ the space of Schwartz distributions,
and by $\mathcal{D}^{*}(K)$ the space of Schwartz distributions supported
on some compact $K$.
\begin{cor*}
If there is an $S\in\mathcal{D}^{*}(K)$ with $\widehat{S}(n)=O(|n|^{-\alpha/2})$,
$n<0$ then $\dim K\ge\alpha$.
\end{cor*}
Somewhat reminiscent results have to do with symmetry properties of
individual measures. For example, a theorem of Rajchman (\cite{R29}
or \cite[\S 1.4]{KL87}) states that for any given complex measure
$\mu$, if $\lim_{n\to\infty}\widehat{\mu}(n)=0$ then it follows
that $\lim_{n\to-\infty}\widehat{\mu}(n)=0$. See \cite{dLK70} for
an interesting generalization. Another interesting result is that
of Hru\v s\v cev and Peller \cite[Corollary 3.26]{HP86}, which state
that for any complex measure $\mu$, if \[
\sum_{n=1}^{\infty}\frac{|\widehat{\mu}(n)|^{2}}{n}<\infty\]
 then also\[
\sum_{n=-\infty}^{-1}\frac{|\widehat{\mu}(n)|^{2}}{|n|}<\infty.\]
Such a result cannot hold with $|n|$ replaced by $|n|^{1-\alpha}$
for some $\alpha>0$. Indeed, \[
\sum_{n=1}^{\infty}\frac{e^{int}}{n^{1-\alpha}}\in L^{1}\quad\forall\alpha>0\]
(see e.g.\ \cite[chap. V]{Z02} where the $\cos$ and $\sin$ terms
are handled independently). Hence one can construct $L^{1}$ examples
with any desired polynomial tail decay on the left and independently
on the right.

\subsection*{Question}

Does this symmetry result still hold if $n$ is replaces by $n\log n$?\medskip{}

\begin{flushleft}
It is occasionally useful to have the following terminology
\par\end{flushleft}
\begin{defn*}
For a Schwartz distribution $S$ we define its analytic part as the
distribution $S'$ satisfying \[
\widehat{S'}(n)=\begin{cases}
\widehat{S}(n) & n\ge0\\
0 & \mbox{otherwise.}\end{cases}\]
The anti-analytic part is $S-S'$.
\end{defn*}

\subsection*{Proof of theorem \ref{thm:main}}

Let us first reduce the general case to the case that $\alpha=1$.
Suppose by contradiction that \begin{equation}
\dim K<\alpha.\label{eq:dimK<}\end{equation}
By convolution with an appropriate measure we want to prepare a new
distribution, $S'$ supported by a compact of dimension $<1$ such
that the anti-analytic part of $S'$ would belong to $L^{2}(\mathbb{T})$.
This will lead to a contradiction. The measure will be taken from
the following result of Salem:
\begin{lem}
\label{Sa:lem}For every $\delta$ , $0<\delta<1$ there is a compact
set $E\subset\mathbb{T}$ such that
\begin{enumerate}
\item $\dim_{\Mink}(E)\leq\delta$;
\item \label{enu:dimF}For every $\beta<\delta$ there exists a positive
measure $\mu_{\beta}$ supported on $E$ such that $|\widehat{\mu_{\beta}}(n)|\le C|n|^{-\beta/2}$.
\end{enumerate}
\end{lem}
This result is essentially due to Salem but all references we could
find were for the usual Hausdorff dimension while we need the Minkowski
dimension. We recall that $\dim_{\Mink}$, the Minkowski dimension
(also known as the upper box dimension) is defined by \[
\dim_{\textrm{Mink}}(E)=\varlimsup_{\epsilon\to0}\frac{\log\Cov(E;[0,\epsilon])}{\log1/\epsilon}\]
where $\Cov(E;A)$ is the cover number of the set $E$ by the set
$A$ i.e.\ the minimal number of translates of $A$ required to cover
$E$. Directly from the definition we see that $\dim_{\Mink}(E)\geq\dim_{\Haus}E$.
Of course, by Frostman's theorem clause \ref{enu:dimF} implies that
$\dim_{\Haus}\ge\delta$ so in fact the set $E$ of lemma \ref{Sa:lem}
satisfies $\dim_{\Haus}(E)=\dim_{\Mink}(E)=\delta$. A common way
to refer to Frostman and Salem's results in the literature is via
the notion of the Fourier dimension \cite[\S 17]{K85}:
\begin{defn*}
\label{def:Fdim}For a compact set $K\subset\mathbb{T}$ the supremum
of numbers $a$ such that $K$ supports a positive (non-zero) measure
$\mu$ with $|\widehat{\mu}(u)|^{2}\le C/|u|^{a}$ is known as the
Fourier dimension of $K$. We denote it by $\dim_{\F}(K)$.
\end{defn*}
With this definition we have the following succinct statements

\emph{Frostman's theorem}: $\dim_{\F}(E)\le\dim_{\Haus}(E)$. 

\emph{Salem's theorem}: For every $0<\delta<1$ there exists some
$E$ with \[
\dim_{\F}(E)=\dim_{\Haus}(E)=\delta.\]

\emph{Lemma \ref{Sa:lem}}: For every $0<\delta<1$ there exists some
$E$ with \[
\dim_{\F}(E)=\dim_{\Haus}(E)=\dim_{\Mink}(E)=\delta.\]

\begin{proof}
[Proof of lemma \ref{Sa:lem}]By \cite[chap.\ 17, thm.\ 4]{K85} the
image of any set of dimension $\delta/2$ by a Brownian bridge is
a Salem set. Recall that a Brownian bridge is $B=W+L$ where $W$
is usual Brownian motion while $L$ is a linear term that makes $B$
continuous on the circle (see e.g.\ \cite[\S 16.3 \& \S 17.5]{K85}
where the Brownian bridge is called {}``the Wiener function'').
Thus it remains to show that the Brownian image of, for example, a
Cantor-like set has the required Minkowski dimension, almost surely.
Now, our Cantor set $K$ can be covered by intervals $I_{1},\dotsc,I_{N}$
($N$ being a power of $2$) of length $N^{-2/\delta}$. By definition,
the image of each $I_{i}$ by Brownian motion is an interval, whose
length has expectation $\le CN^{-1/\delta}$. Let $X_{i}$ be the
number of intervals of length $N^{-1/\delta}$ required to cover $WI_{i}$.
Then $X_{i}$ are i.i.d.\ random variables with expectation some
constant $C$. We get \[
\mathbb{E}\Cov\bigg(\bigcup_{i=1}^{N}WI_{i};[0,N^{-1/\delta}]\bigg)\le\sum_{i}\mathbb{E}X_{i}=NC\]
The linear term $L$ is negligible in comparison: if $\Cov(BI_{i};[0,N^{-1/\delta}])>\Cov(WI_{i};\linebreak[0][0,\linebreak[0]N^{-1/\delta}])+1$
then this means that the derivative of the linear term must be $>N^{-1/\delta}$
which can only happen for a finite number of $N$. Hence for all $N$
sufficiently large \[
\mathbb{E}\Cov(BK;[0,N^{-1/\delta}])\le\mathbb{E}\Cov\bigg(\bigcup_{i=1}^{N}BI_{i};[0,N^{-1/\delta}]\bigg)\le NC\]
and by Tchebyshev's inequality, for any $\eta>0$,\[
\mathbb{P}(\Cov(BK;[0,N^{-1/\delta}])>N^{1+\eta})\le CN^{-\eta}.\]
Summing this over $N=2^{n}$ for all $n$ we get that with probability
$1$, for all $N$ sufficiently large $\Cov(BK;[0,N^{-1/\delta}])\le N^{1+\eta}$.
By the monotonicity of the covering numbers we can move from a specific
sequence of $\epsilon$, $\epsilon=2^{-n/\delta}$ to a general $\epsilon$
and lose only a constant. Hence\[
\Cov(BK;[0,\epsilon])\le C\epsilon^{-\delta-\eta\delta}\]
So $\dim_{\Mink}(BK)\le\delta+\eta\delta$. Since $\eta>0$ was arbitrary,
the lemma is proved.
\end{proof}
We remark that we use Kahane's approach to the construction of a Salem
set for convenience only. In fact one may check that Salem's original
approach (\cite[chap.\ VIII]{KS94} or \cite[\S 3]{M00}) as well
as Kaufman's \cite{K81} also give sets with the correct Minkowski
dimension.

We need to consider the Minkowski dimension due to the following lemma 
\begin{lem}
\label{lem:K+E}For any compacts $K$, $E$, \[
\dim_{\Haus}(K+E)\leq\dim_{\Haus}K+\dim_{\Mink}(E).\]

\end{lem}
(The Minkowski dimension on the right hand side can not be replaced
by the Hausdorff one, see e.g.~\cite[example 7.8]{F03}).
\begin{proof}
By \cite[\S 8.10 (4)]{M95}, \[
\dim_{\Haus}(K\times E)\le\dim_{\Haus}(K)+\dim_{\Mink}(E)\]
and the set $K+E$ is a Lipschitz map of $K\times E$.
\end{proof}
Returning to the proof of Theorem \ref{thm:main} take\begin{equation}
1-\dim K>\delta>1-\alpha\label{eq:defdelta}\end{equation}
and find $E$ and $\mu$ from lemma \ref{Sa:lem} with $\beta_{\textrm{lemma \ref{Sa:lem}}}=1-\alpha$.
What we would like to do next is to define a distribution $T$ by
$T=S*\mu$. However, we have to consider the unlikely situation that
this convolution is $0$. Definitely, though, we may find some $m$
such that $T=S*(e^{imt}\mu)$ is not zero and we define $T$ this
way. $e^{imt}\mu$ still satisfies the crucial property of $\mu$
i.e.\ $|\widehat{e^{imt}\mu}(n)|\le C|n|^{(\alpha-1)/2}$. We get
\[
T\in\mathcal{D}^{*}(K')\mbox{ where }K':=K+E.\]
Due to  (\ref{eq:defdelta}) and lemma \ref{lem:K+E} we get: \[
\dim K'<1.\]
(when we do not state which dimension it is, we mean the Hausdorff
dimension). On the other hand,\[
\sum_{n<0}|\widehat{T}(n)|^{2}\le\sum_{n<0}|\widehat{S}(n)|^{2}\cdot C|n|^{\alpha-1}<\infty.\]
In other words we have reduced the problem to the case that $\alpha=1$.
Hence the following lemma concludes the proof:
\begin{lem}
\label{lem:Berman}Let $K$ be a compact with the Hausdorff measure
$\Lambda_{x\log1/x}(K)=0$. In particular, this is satisfied if $\dim K<1$.
Let $S\in\mathcal{D}^{*}(K)$. If $\widehat{S}(n)\in l^{2}(\mathbb{Z}^{-})$
then $S=0$.
\end{lem}
(the notation $\widehat{S}(n)\in l^{2}(\mathbb{Z}^{-})$ is short
for $\sum_{n<0}|\widehat{S}(n)|^{2}<\infty$). We remind the reader
that the Hausdorff measure $\Lambda_{h}$ for an increasing function
$h$ with $h(0)=0$ is the natural generalization of the usual $\alpha$-Hausdorff
measure ones gets by examining $\sum h(\diam B_{i})$ for a covering
of a set by balls $B_{i}$. See e.g.\ \cite[\S 10.2, remark 4]{K85}.

Lemma \ref{lem:Berman} is a consequence of results by Dahlberg \cite{D77}
and Berman \cite{B92}. For the convenience of the reader, let us
quote the relevant result from \cite{B92} (almost) literally and
then show how it implies lemma \ref{lem:Berman}.\renewcommand{\theenumi}{(\alph{enumi})} \renewcommand{\labelenumi}{\theenumi}

\begin{NIH}Let $\omega:\mathbb{R}^{+}\to\mathbb{R}^{+}$ satisfy
that $\omega(x)/x$ is decreasing and $\omega'(0)=\infty$. Let $F$
be an analytic function on the disk $\mathbb{D}=\{|z|<1\}$ satisfying\begin{equation}
\log|F(z)|\leq C\frac{\omega(1-|z|)}{1-|z|}.\label{eq:BermanGrowth}\end{equation}
Let $K\subset\partial\mathbb{D}$ be a set satisfying. 
\begin{enumerate}
\item $\Lambda_{\omega}(K)=0$\label{enu:BermanDim}
\item \label{enu:BermanLimit}For any $\zeta\in\partial\mathbb{D}\setminus K$
there exists a path $\rho$ in $\mathbb{D}$ ending in $\zeta$ such
that $F$ is bounded on $\rho$.
\end{enumerate}
Assume $g\in L^{1}(\mathbb{T})$ and\begin{equation}
\liminf_{r\to1}\log|F(re^{it})|\leq g(t)\quad\textrm{a.e.}\label{eq:BermanFG}\end{equation}
Then $\log|F(z)|\leq G(z)$ for all $z\in\mathbb{D}$, where $G(z)$
is the harmonic extension of $g$ to $\mathbb{D}$.\end{NIH}\renewcommand{\theenumi}{(\roman{enumi})} \renewcommand{\labelenumi}{\theenumi}

See \cite{B92}, the corollary to theorem 5 on page 479. Actually,
the theorem there is more general: all conditions are necessary only
on an arc $\gamma$ (except for a very mild global growth condition)
and the result is that $\log|F|\le G$ in a neighborhood of $\gamma$.
\begin{proof}
[Proof of lemma \ref{lem:Berman}]Examine the analytic function \[
F(z)=\sum_{n=0}^{\infty}\widehat{S}(n)z^{n}.\]
We need two properties of $F$:
\begin{enumerate}
\item \label{enu:Fcont}For every $u\not\in K$, the limit $\lim_{z\to e^{iu}}F(z)$
exists and is finite. 
\item \label{enu:fL2}The (a.e.~defined) limit function $f(u)=\lim_{z\to e^{iu}}F(z)$
is in $L^{2}(\mathbb{T})$.
\end{enumerate}
We remark on \ref{enu:Fcont} that we need much less --- we only need
that the limit exists along some path $\rho$, e.g.\ radially. But
as we will see, the function $F$ is continuous in a neighborhood
of $u$. 

To show property \ref{enu:Fcont} write, for $z\in\mathbb{D}$ \[
F(z)=\sum_{n=0}^{\infty}z^{n}\int S(t)e^{-int}dt=\int S(t)Q_{z}(t)\, dt\]
where $Q_{z}=\sum_{0}^{\infty}z^{n}e^{-int}=\frac{1}{1-ze^{-it}}$
(this is the kernel of the Cauchy transform, which is also the sum
of the Poisson kernel and the conjugate Poisson kernel, but this is
not important at this point). Note that because $S$ is a distribution,
the notation $S(t)$ is slightly misleading --- the integral $\int S(t)Q_{z}(t)\, dt$
though is well defined. Let $\psi$ be a $C^{\infty}$ function supported
on $[u-\epsilon,u+\epsilon]$ with $\psi\equiv1$ on $[u-\frac{1}{2}\epsilon,u+\frac{1}{2}\epsilon]$,
where $\epsilon=\frac{1}{2}\dist(u,K)$. Then\[
\int S(t)Q_{z}(t)\psi(t)\, dt=0\]
because $\psi$ is $0$ on a neighborhood of $K$ and $S$ is supported
on $K$. On the other hand, \[
\lim_{z\to e^{iu}}\int S(t)Q_{z}(t)(1-\psi(t))\, dt=\int S(t)\left(Q_{e^{iu}}(t)(1-\psi(t))\right)\, dt\]
since after multiplication by $1-\psi(t)$ we get that $Q_{z}(1-\psi(t))$
is smooth and its derivatives depend continuously on $z$ all the
way up to $e^{iu}$. This shows that the limit function $f$ exists
and is finite outside $K$ i.e.\ proves property \ref{enu:Fcont}.

For property \ref{enu:fL2} examine the usual Poisson kernel \[
P_{z}(t)=\sum_{n=0}^{\infty}z^{n}e^{-int}+\sum_{n=-\infty}^{-1}\bar{z}^{n}e^{int}=\frac{1-|z|^{2}}{1-2|z|\cos(t-\arg z)+|z|^{2}}.\]
Any distribution is Abel summable to $0$ outside its support \cite[Proposition 2, appendix I, p. 162]{KS94}
which means that for any $u\not\in K$,\[
\lim_{r\to1}\int S(t)P_{re^{iu}}(t)=0.\]
and hence \[
f(u)=\lim_{r\to1}\int S(t)\left(Q_{re^{iu}}(t)-P_{re^{iu}}(t)\right)\, dt.\]
However, $Q-P=-\sum_{1}^{\infty}\bar{z}^{n}e^{int}$ so we get\[
f(u)=\lim_{r\to1}\sum_{n=1}^{\infty}-\widehat{S}(-n)\left(re^{-iu}\right)^{n}\]
But $\sum_{1}^{\infty}|\widehat{S}(-n)|^{2}<\infty$! We get that
$f\in L^{2}$, showing property \ref{enu:fL2}.

We now apply Berman's theorem. Let $\omega(x)=x\log(2/x)$. We first
note that $F$ satisfies (\ref{eq:BermanGrowth}) because\[
\log|F(z)|=\log\left(\sum_{n=0}^{\infty}\widehat{S}(n)z^{n}\right)\le\log\left(\sum_{n=0}^{\infty}(C+n^{C})|z|^{n}\right)\le C\log\left(\frac{1}{1-|z|}\right)\]
(the middle inequality holds because any Schwarz distribution $S$
has that $\widehat{S}$ grows no faster than polynomially). Next,
$K$ satisfies \ref{enu:BermanDim} by assumption and \ref{enu:BermanLimit}
from property \ref{enu:Fcont} which shows that we may take $\rho$
to be simply the radius. Define now \[
g(t)=\max\{0,\log|f(t)|\}.\]
$g$ is in $L^{1}$ because $f\in L^{2}$ --- we cannot control $(\log f)^{-}$
but $(\log f)^{+}$ is definitely well-behaved. We get (\ref{eq:BermanFG})
again from the continuity of $F$. Hence Berman's theorem holds and
we get $\log|F|\leq G$, where $G$ is as in Berman's theorem namely
the harmonic extension of $g$ to $\mathbb{D}$. To estimate $G$
we note that by definition $G(z)=\int P_{z}(t)g(t)\, dt$, and hence
by Jensen's inequality,\begin{equation}
e^{2G(z)}=\exp\Big(\int_{\mathbb{T}}P_{z}(t)2g(t)\, dt\Big)\leq\int_{\mathbb{T}}P_{z}(t)e^{2g(t)}\, dt.\label{eq:Jensen}\end{equation}
This observation gives that $F$ is in the Hardy space $H^{2}$ because
\begin{align*}
\int_{|z|=r}|F(z)|^{2}\, d|z| & \le\int_{|z|=r}e^{2G(z)}\, d|z|\stackrel{(\ref{eq:Jensen})}{\le}\int_{|z|=r}\int_{\mathbb{T}}P_{z}(t)e^{2g(t)}\, dt\, d|z|=\\
 & =\int_{\mathbb{T}}e^{2g(t)}\, dt\le2\pi+\int_{\mathbb{T}}|f(t)|^{2}\, dt<\infty.\end{align*}
This however is impossible: if $F\in H^{2}$ then by the definition
of $F$ we get that $\sum_{n=0}^{\infty}|\widehat{S}(n)|^{2}<\infty$
and since we already know that for the negative Fourier coefficients
we get that $S$ is an $L^{2}$ function. But it is supported on $K$
which has zero measure. Hence it is identically $0$.
\end{proof}

\subsection*{Question}

Frostman's theorem for the case $\alpha=0$ states that a compact
$K$ supports a distribution $S$ with \[
\sum_{n}\frac{|\widehat{S}(n)|^{2}}{|n|+1}<\infty\]
if and only if it has positive logarithmic capacity \cite[chapter III, theorem V]{KS94}.
Does this theorem have a one-sided analog? We recall again the theorem
of Hru\v s\-\v cev and Peller which states that for any measure $\mu$,
the one-sided estimate $\sum_{n>0}|\widehat{\mu}(n)|^{2}/n<\infty$
implies the two-sided estimate above. So the question is really only
about distributions.

\section*{II}

There is a delicate difference between the two-sided and one-sided
results. The condition (\ref{eq:twosided}) in fact implies that the
compact $K$ has positive $\alpha$-measure, see \cite[\S10.3, theorem 2, p. 132]{K85}.
The condition (\ref{eq:one-sided}) does not (at least for $\alpha=1$).
Indeed, the following, quite surprising, result was proved in \cite{KO.03}
\begin{thm}
\label{thm:KO03}There exist a singular (non-trivial) pseudo-function
$S$ such that $\widehat{S}\in l^{2}(\mathbb{Z}^{-})$.
\end{thm}
(recall that a pseudo-function is a distribution with $\lim_{|n|\to\infty}\widehat{S}(n)=0$).
This is surprising because it follows that the $L^{2}$ function $f(t)=\sum_{n<0}\widehat{S}(n)e^{int}$
has an a.e.\ converging representation by an analytic sum which is
different from its Fourier expansion. However, such representations
are unique (by Privalov's theorem). Thus theorem \ref{thm:KO03} contradicted
a long-established view that \emph{if a method of representation by
a trigonometric expansion is unique, it must be the Fourier expansion}.
We named the space of functions that have an analytic representation
of this sort PLA. See \cite{KO.03} for more details.

Theorem \ref{thm:KO03} can be strengthened to show that the anti-analytic
amplitudes of S may vanish faster then any power of $1/|n|$, see
\cite{KO.04}, where the precise threshold for the decay is found.
The class PLA was further analyzed in \cite{KO.07}.

Our goal now is to find the threshold for the size of support in Theorem
\ref{thm:KO03} in terms of Hausdorff $\Lambda$-measure. By lemma
\ref{lem:Berman} we see that if $\Lambda_{x\log1/x}(K)=0$, no such
example may exist, and the condition that the positive Fourier coefficients
tend to $0$ is not needed for this (the assumption that $S$ is a
Schwartz distribution does imply that the positive coefficients do
not grow faster than polynomial). We now show that this condition
is precise.
\begin{thm}
\label{thm:thickness}There exist a non-trivial pseudo-function $S$
with $\widehat{S}(n)\in l^{2}(\mathbb{Z}^{-})$\noun{ }supported on
a compact of finite $t\log1/t$-Hausdorff measure.
\end{thm}
The idea of the proof is as follows. Take the natural measure $\mu$
on some symmetric compact of the exact size. Extend it to a harmonic
function on the disk, and let $F=e^{\mu+i\widetilde{\mu}}$ be an
analytic function inside the disk. Let $f$ be its almost-everywhere
boundary limit. $S=F-f$ is the required object, except it might not
be a pseudo-function. To solve this problem we introduce a random
perturbation in the compact (and hence in $\mu$, $F$, $f$ and $S$).
If one does not require from $S$ to be a pseudo-function then the
probabilistic part of the proof is not needed. We shall therefore
note in the proof what is needed to get the simpler result, which
we shall refer to as {}``the corollary'':
\begin{cor*}
There exist a non-trivial Schwartz distribution $S$ with $\widehat{S}(n)\in l^{2}(\mathbb{Z}^{-})$\noun{
}supported on a compact of finite $t\log1/t$-Hausdorff measure.
\end{cor*}

\subsection*{Question}

What about the analog for $\alpha<1$? Namely, is it true that for
any $\alpha<1$ there exists a pseudo-function $S$ supported on a
set of zero $\Lambda_{x^{\alpha}}$-Hausdorff measure with \[
\sum_{n<0}\frac{\left|\widehat{S}(n)\right|^{2}}{|n|^{1-\alpha}}<\infty?\]

\noindent \emph{Proof} \emph{of theorem \ref{thm:thickness} and of
the corollary}. As mentioned above, our construction is similar to
that of \cite{KO.03} in that we take a singular measure $\mu$ and
define $F=e^{\mu+i\widetilde{\mu}}$. For the theorem we will take
$\mu$ to be random, an idea used also in \cite{KO.04}, to get improved
smoothness of the Fourier transform, but for the corollary that will
not be needed. Here are the details. Define \begin{equation}
\sigma_{n}=\frac{2\pi}{n2^{n}},\quad\tau_{n}=\frac{1}{6}(\sigma_{n-1}-2\sigma_{n})\quad n\geq0.\label{eq:deftaun}\end{equation}
Hence, \begin{equation}
\frac{\tau_{n}}{\sigma_{n}}=\frac{1}{3(n-1)}.\label{eq:tausigmanu3}\end{equation}
Let $s(n,k)$ be a collection of numbers between $0$ and $1$, for
each $n\in\mathbb{N}$ and each $0\leq k<2^{n}$. Most of the proof
will hold for any choice of $s(n,k)$, but in the last part we shall
make them random, and prove that the constructed function will have
the required properties for almost any choice of $s(n,k)$. For the
corollary one may take them to be $0$. Define now inductively intervals
$I(n,k)=[a(n,k),a(n,k)+\sigma_{n}]$ (we call these $I(n,k)$ {}``intervals
of rank $n$'') using the following: $I(0,0)=[0,2\pi]$ and for $n\geq0$,
$0\leq k<2^{n}$ \begin{align}
a(n+1,2k) & =a(n,k)+\tau_{n+1}(1+s(n+1,2k))\nonumber \\
a(n+1,2k+1) & =a(n,k)+{\textstyle \frac{1}{2}}\sigma_{n}+\tau_{n+1}(1+s(n+1,2k+1))\label{eq:defank}\end{align}
In other words, at the $n^{\textrm{th}}$ step, inside each interval
of rank $n$ (which has length $\sigma_{n}$), situate two disjoint
intervals of rank $n+1$ of lengths $\sigma_{n+1}$ in random places
(but not too near the boundary of $I(n,k)$ or its middle). Define
\begin{align*}
K & :=\bigcap_{n=1}^{\infty}K_{n}, & K_{n} & :=\bigcup_{k=0}^{2^{n}-1}I(n,k).\\
K^{\circ} & :=e^{iK} & K_{n}^{\circ} & :=e^{iK_{n}}.\end{align*}
Note that $|K_{n}|=\frac{2\pi}{n}$ and hence $K$ has zero measure
and finite $\Lambda_{t\log1/t}$-measure. Define a measure $\mu$
by the weak limit of the measures with density $\frac{1}{|K_{n}|}\mathbf{1}_{K_{n}}$,\[
\mu=\lim_{n\to\infty}\frac{1}{|K_{n}|}\mathbf{1}_{K_{n}}.\]
It is easy to see that the limit exists and is supported on $K$.
However, we will need later on a more quantitative version of this
convergence. Define therefore \[
g_{n}=\frac{1}{|K_{n}|}\mathbf{1}_{K_{n}}\]
and let $G_{n}$ be the harmonic extension of $g_{n}$ to the disk
$\overline{\mathbb{D}}$.
\begin{lem}
\label{lem:GnG}For any $z\in\overline{\mathbb{D}}\setminus K_{n}^{\circ}$\begin{equation}
|G_{n+1}(z)-G_{n}(z)|\leq\frac{C}{2^{n}d(z,K_{n}^{\circ})}.\label{eq:fnDfn1DdzK}\end{equation}
Further, this holds also for the conjugate harmonic functions $\widetilde{G_{n}}$,
\begin{equation}
|\widetilde{G_{n+1}}(z)-\widetilde{G_{n}}(z)|\leq\frac{C}{2^{n}d(z,K_{n}^{\circ})}.\label{eq:zntilde}\end{equation}
\end{lem}
\begin{proof}
For any $n$ and $k$, $\int_{I(n,k)}g_{n}=2^{-n}$. Subtracting we
get \begin{equation}
\int_{I(n,k)}\left(g_{n+1}(x)-g_{n}(x)\right)\, dx=0.\label{eq:fnfnmin1zero}\end{equation}
and because $g_{n+1}-g_{n}$ is non-zero only on the intervals $I(n,k)$,\begin{equation}
\Big|\int_{t}^{u}g_{n+1}(x)-g_{n}(x)\, dx\Big|\leq2^{-n}\quad\forall t,u\in[0,1],\,\forall n\label{eq:intgparts}\end{equation}
Write $G(z)=\int_{\mathbb{T}}g(t)P_{z}(t)$, where $P_{z}$ is the
Poisson kernel. We divide into two cases: if $1-|z|>\frac{1}{2}d(z,K_{n}^{\circ})$
then\[
\int_{\mathbb{T}}|P_{z}^{'}|\leq\frac{C}{(1-|z|)}\leq\frac{C}{d(z,K_{n}^{\circ})}.\]
(the first inequality is a well known property of the Poisson kernel).
On the other hand, if $1-|z|\leq\frac{1}{2}d(z,K_{n}^{\circ})$ then
$g_{n+1}-g_{n}$ is zero in an interval $J:=[t-cd(z,K_{n}^{\circ}),\linebreak[0]t+cd(z,K_{n}^{\circ})]$
for some $c$ sufficiently small, where $t$ is given by $e^{it}=z/|z|$,
and \[
\int_{\mathbb{T}\setminus J}|P_{z}^{'}|\leq\frac{C}{d(z,K_{n}^{\circ})}.\]
In either case , a simple integration by parts gives (\ref{eq:fnDfn1DdzK})
on $\mathbb{D}$. Finally, on $\partial\mathbb{D}\setminus K$ we
have $G_{n+1}(e^{it})-G_{n}(e^{it})=g_{n+1}(t)-g_{n}(t)=0$ for every
$t\not\in K_{n}$.

The proof for $\widetilde{G_{n}}$ is identical except the Poisson
kernel $P_{z}$ has to be replaced with the conjugate kernel $Q_{z}$.
\end{proof}
Let $G$ be the harmonic extension of $\mu$ into the unit disc i.e.\ the
Poisson transform of $\mu$ (equivalently you may define $G$ as the
limit of the $G_{n}$). Let $\widetilde{G}$ be the harmonic conjugate
of $G$. Let $\delta\in(0$,1) be some sufficiently small number (we
will fix its value later). Let $F=\exp\left(\delta(G+i\widetilde{G})\right)$.
We note the following properties of $F$
\begin{enumerate}
\item $F$ is unbounded in the unit disc.
\item $F$ has a boundary limit at almost every point of the boundary of
the disk, and this limit is uniform on every closed interval disjoint
from $K^{\circ}$. Denote the boundary limit by $f$,\[
f(t):=\lim_{z\to e^{it}}F(z).\]

\item The function $f$ is bounded.
\end{enumerate}
All properties are simple. The first follows from the fact that for
any $t\in K$, $\lim_{z\to e^{it}}G(z)=\infty$ and $|F(z)|=e^{\delta G(z)}$.
The second follows because $G$ and $\widetilde{G}$ have boundary
limits outside of $K$ --- recall that $G$ is the harmonic extension
of a measure $\mu$ supported on the compact $K$. The third follows
because for any $t\not\in K$ $\lim_{z\to e^{it}}G(z)=0$ and hence
$|f(t)|=1$.

We now define our distribution $S$ by {}``$F-f$'' or formally
by\[
\widehat{S}(m)=\widehat{F}(m)-\widehat{f}(m)\]
where $\widehat{F}$ are the Taylor coefficients of the analytic function
$F$ at $0$ namely\[
F(z)=\sum_{m=0}^{\infty}\widehat{F}(m)z^{m}\qquad\widehat{F}(-m)=0\,\forall m\in\mathbb{N}\]
and $\widehat{f}$ are the Fourier coefficients of $f$. Because $\widehat{F}(-m)=0$
and $f\in L^{2}$ we immediately get that $\widehat{S}\in l^{2}(\mathbb{Z}^{-})$.
The uniform convergence of $F$ outside $K$ shows that $S$ is supported
on $K$. Thus, for the corollary it is enough to show that $S$ is
a Schwartz distribution i.e.\ that $\widehat{S}(m)\le Cm^{C}$ for
some constant $C$, while for the theorem it is necessary to show
that $S$ is a pseudo-function i.e.\ that \[
\lim_{m\to\infty}\widehat{S}(m)=0.\]
Since this holds for $f$, we need only verify that $\widehat{F}(m)\to0$,
with probability $1$.

\subsection*{\noindent \label{sub:Taylor}The Taylor coefficients of $F$}

\noindent We will now show that with probability 1, $\widehat{F}(m)\rightarrow0$
as $m\rightarrow\infty$. The first step is to define $F_{n}=e^{\delta(G_{n}+i\widetilde{G_{n}})}$
and find some $n$ such that $\widehat{F_{n}}(m)$ approximates $\widehat{F}(m)$.
Summing (\ref{eq:fnDfn1DdzK}) and (\ref{eq:zntilde}) over $n$ we
get\[
|(G_{n}+i\widetilde{G_{n}})(z)-(G+i\widetilde{G})(z)|\leq\frac{C}{(1-|z|)2^{n}}.\]
Fix, therefore, $n=n(m):=\left\lceil C\log m\right\rceil $ for some
$C$ sufficiently large, and get, for every $z$ with $|z|=1-\frac{1}{m}$
that $|(G_{n}+i\widetilde{G_{n}})(z)-(G+i\widetilde{G})(z)|\leq1/m$.
Now, \[
\sup_{z\in\overline{\mathbb{D}}}|G_{n}(z)|=\sup_{t\in[0,2\pi]}|g_{n}(t)|=\frac{1}{|K_{n}|}=n\]
so $|F_{n}(z)|\le e^{\delta n}$. Hence for $|z|=1-\frac{1}{m}$,
\begin{equation}
|F_{n}(z)-F(z)|\leq|F_{n}(z)||1-\exp(\delta((G_{n}+i\widetilde{G_{n}})(z)-(G+i\widetilde{G})(z)))|\leq C\frac{e^{\delta n}}{m}\label{eq:Fn-F}\end{equation}
and if $\delta$ is taken sufficiently small, this is $\le Cm^{-1/2}$.
Finally we use \[
\widehat{F}(m)=\frac{1}{m!}F^{(m)}(0)=\int_{|z|=1-1/m}z^{-m-1}F(z)\, dz\]
so\begin{equation}
|\widehat{F_{n}}(m)-\widehat{F}(m)|=\left|\int_{|z|=1-1/m}z^{-m-1}(F_{n}(z)-F(z))\, dz\right|\leq Cm^{-1/2}\label{eq:FnFmhalf}\end{equation}
and we see that it is enough to calculate $\widehat{F_{n}}(m)$. At
this point the proof of the corollary is complete. Indeed, we may
write\[
\widehat{F_{n}}(m)\le\left\Vert F_{n}\right\Vert _{2}\le e^{\delta n}\le m^{\delta C}\]
and hence $\widehat{F}(m)\le Cm^{C}$, $S$ is a Schwartz distribution
and the corollary is proved. From now on we focus on proving the theorem.

Since $F_{n}$ is a bounded function on the disk, its Taylor coefficients
at $0$ are identical to the Fourier coefficients of its boundary
value. Denote the boundary value by $f_{n}(t)=F_{n}(e^{it})$. Thus
we have reduced our problem to that of estimating \[
\widehat{f_{n}}(m)=\int f_{n}(t)e^{-imt}\, dt.\]

Now, $f_{n}=e^{\delta(g_{n}+i\widetilde{g_{n}})}$ so on $K_{n}$
it is large ($|f_{n}(t)|=e^{\delta n}$ for any $t\in K_{n}$) while
outside of $K_{n}$ it has absolute value $1$. Let us first show
that the part outside of $K_{n}$ is irrelevant. As in (\ref{eq:Fn-F})
above we use lemma \ref{lem:GnG} and sum over $n$. We get\[
|f_{n}(t)-f(t)|\le\min\bigg(2,\frac{C}{2^{n}d(t,K)}\bigg)\quad\forall t\not\in K_{n}\]
and by integrating\[
\int_{[0,2\pi]\setminus K_{n}}|f_{n}(t)-f(t)|\, dt\le\int_{0}^{2\pi}\min\bigg(2,\frac{C}{2^{n}d(t,K)}\bigg)\, dt\xrightarrow[n\to\infty]{}0\]
since $K$ is a compact of measure zero. The convergence above is
uniform in the choice of the translations $s(n,k)$ that we used to
construct $K$. Since $\widehat{f}(n)\to0$, $f$ being a bounded
function, we get that it is enough to show\begin{equation}
\int_{K_{n}}f_{n}(t)e^{-imt}\, dt\to0\label{eq:intfnKn}\end{equation}
where the limit is as $m\to\infty$ and $n=\left\lceil C\log m\right\rceil $.
Here is where we will use that $K_{n}$ was a random set, and we will
show that this convergence holds for almost every choice of $s(n,k)$.

\subsection*{\label{sub:Probability}Probability}

Take $s(n,k)$ to be independent and uniformly distributed on $[0,1]$.
We shall estimate the integral (\ref{eq:intfnKn}) by moment methods.
Unfortunately, it seems we need the fourth moment. We start with a
lemma that contains the calculation we need without referring to analytic
functions
\begin{lem}
\label{lem:real}Let $I_{i}$ be $4$ intervals and let $\tau,\alpha,\beta>0$
be some numbers. Let $h_{1},h_{2},h_{3}$ be functions satisfying
\begin{align}
\int_{I_{i}}|h_{j}| & \leq\alpha, & \righteqn{i=j\mbox{ or }i=4\mbox{ and }j=3}\label{eq:estGiIi}\\
|h_{j}(x)| & =1,|h_{j}'(x)|\leq\beta,|h_{j}''(x)|\leq\beta^{2}\quad\forall x\in I_{i}+[-\tau,\tau] & \mbox{otherwise}\label{eq:estGtg}\end{align}
where {}``$+$'' stands for regular set addition. See table 1. Let
$t_{1}$ and $t_{2}$ be two random variables, uniformly distributed
on $[0,\tau]$, and let $t_{3}=t_{4}=0$. Define\begin{equation}
f(x)=f_{t_{1},t_{2}}(x):=h_{1}(x-t_{1})h_{2}(x-t_{2})h_{3}(x).\label{eq:defF}\end{equation}
Then\begin{equation}
E:=\left|\mathbb{E}\left(\prod_{i=1}^{4}\int_{I_{i}+t_{i}}f(x_{i})e^{-imx_{i}}\, dx_{i}\right)\right|\leq C\frac{\alpha^{4}}{m^{2}}\left(\max\beta,\frac{1}{\tau}\right)^{2}.\label{eq:defE}\end{equation}
\end{lem}
\begin{proof}
See \cite[lemma 9, page 1050]{KO.04}. The proof is essentially nothing
more than two integrations by parts, each one giving a $\frac{1}{m}$
factor (from integrating $e^{-imx}$) and a $\beta$ factor from differentiating
the $h$-s. The $\frac{1}{\tau}$ factors come from the boundary conditions.
\end{proof}
\begin{table}
\begin{centering}
\begin{tabular}{|c|c|c|c|c|}
\cline{2-5} 
\multicolumn{1}{c|}{} & $I_{1}$ & $I_{2}$ & $I_{3}$ & $I_{4}$\tabularnewline
\hline 
$h_{1}$ & ${\displaystyle \vphantom{\int_{I_{M_{M}}}^{M}}}{\displaystyle \int_{I_{1}}|h_{1}|\le\alpha}$ & $h_{1}^{(k)}\le\beta^{k}$ & $h_{1}^{(k)}\le\beta^{k}$ & $h_{1}^{(k)}\le\beta^{k}$\tabularnewline
\hline 
$h_{2}$ & $h_{2}^{(k)}\le\beta^{k}$ & ${\displaystyle \vphantom{\int_{I_{M_{M}}}^{M}}}{\displaystyle \int_{I_{2}}|h_{2}|\le\alpha}$ & $h_{2}^{(k)}\le\beta^{k}$ & $h_{2}^{(k)}\le\beta^{k}$\tabularnewline
\hline 
$h_{3}$ & $h_{3}^{(k)}\le\beta^{k}$ & $h_{3}^{(k)}\le\beta^{k}$ & ${\displaystyle \vphantom{\int_{I_{M_{M}}}^{M}}}{\displaystyle \int_{I_{3}}|h_{3}|\le\alpha}$ & ${\displaystyle \int_{I_{4}}|h_{3}|\le\alpha}$\tabularnewline
\hline
\end{tabular}
\par\end{centering}

\caption{Relations between the functions $h$ and the intervals $I$. $k$
is 0, 1 or 2.}

\end{table}
 Continuing the proof of the theorem, for every $0\leq k<2^{n}$
denote \[
\mathcal{I}_{k}=\int_{I(n,k)}f_{n}(x)e^{-imx}\, dx\]
Recall that on $I(n,k)$ $|f_{n}|=e^{\delta n}$, so \begin{equation}
|\mathcal{I}_{k}|\leq\int_{I(n,k)}|f_{n}(x)|=\frac{2\pi}{n2^{n}}e^{\delta n}=:\alpha.\label{eq:defalpha}\end{equation}
In other words, $\alpha=\alpha(n)$ is a bound for $|\mathcal{I}_{k}|$
independent of $k$.
\begin{lem}
\label{lem:complex}Let $0\leq k_{1},k_{2},k_{3},k_{4}<2^{n}$ and
let $1\leq r<n$, and assume that the $I(n,k_{i})$ belong to at least
three different intervals of rank $r$. Then \[
\left|\mathbb{E}(\mathcal{I}_{k_{1}}\mathcal{I}_{k_{2}}\mathcal{I}_{k_{3}}\mathcal{I}_{k_{4}})\right|\leq\alpha^{4}\frac{Cn^{2}}{m^{2}\tau_{r}^{3}}.\]

\end{lem}
(recall that $\tau_{r}$ were defined in (\ref{eq:deftaun}))
\begin{proof}
Define $q_{1},\dotsc,q_{4}$ using $I(n,k_{i})\subset I(r,q_{i})$.
We may assume without loss of generality that the two $q_{i}$-s which
may be equal are $q_{3}$ and $q_{4}$. Let $\mathcal{X}$ be the
$\sigma$-field spanning all $s$ \emph{except} $s(r,q_{1})$ and
$s(r,q_{2})$. We shall show\[
\left|\mathbb{E}(\mathcal{I}_{k_{1}}\mathcal{I}_{k_{2}}\mathcal{I}_{k_{3}}\mathcal{I}_{k_{4}}|\mathcal{X})\right|\leq\alpha^{4}\frac{Cn^{2}}{m^{2}\tau_{r}^{3}}\]
and then integrating over $\mathcal{X}$ will give the result. We
note that conditioning on $\mathcal{X}$ is in effect fixing everything
except the positions of $I(r,q_{1})$ and $I(r,q_{2})$ inside $I(r-1,\left\lfloor q_{i}/2\right\rfloor )$.
Denote $J_{j}:=I(r,q_{j})$ ($j=1,2$) and $J_{3}=\mathbb{T}\setminus(J_{1}\cup J_{2})$.
Assume for a moment that $s(r,q_{1})=s(r,q_{2})=0$ and define, using
this assumption,\begin{equation}
\begin{aligned}\eta_{j} & :=g_{n}|_{J_{j}} & j & =1,2,3, & h_{j} & :=e^{\eta_{j}+i\widetilde{\eta}_{j}},\\
I_{i} & :=I(n,k_{i}) & i & =1,2,3,4.\end{aligned}
\label{eq:defIi}\end{equation}
Under the assumption $s(r,q_{1})=s(r,q_{2})=0$ we clearly have $f_{n}=h_{1}h_{2}h_{3}$
and when we remove this assumption, the only change is a translation
of $h_{1}$ and $h_{2}$. In other words, if we define $t_{i}=s(r,q_{i})\tau_{r}$
then $f_{n}(x)=h_{1}(x-t_{1})h_{2}(x-t_{2})h_{3}(x)$. Examining (\ref{eq:defF})
we see that $|\mathbb{E}(\mathcal{I}_{k_{1}}\mathcal{I}_{k_{2}}\mathcal{I}_{k_{3}}\mathcal{I}_{k_{4}}|\mathcal{X})|=E$
where $E$ is defined by (\ref{eq:defE}); where the $I_{i}$ of (\ref{eq:defE})
are the same as those of (\ref{eq:defIi}); and where the $\tau$
of (\ref{eq:defE}) is $\tau_{r}$ and where the $\alpha$ of (\ref{eq:estGiIi})
is our $\alpha$. To make (\ref{eq:defE}) concrete we need to specify
a value for the $\beta$ of (\ref{eq:estGtg}) and prove that it holds.
We define\[
\beta=C_{1}\frac{n}{\tau_{r}^{3/2}}\]
for some sufficiently large constant $C_{1}$ to be fixed later. Notice
that $\beta$ is obviously larger than $1/\tau_{r}$. With all these,
lemma \ref{lem:complex} would follow from lemma \ref{lem:real} once
we show (\ref{eq:estGtg}).

Examining the definitions of $\eta_{j}$ and $I(n,k)$ it is easy
to see that $\eta_{j}(x)=0$ for $x\in I_{i}+[-\tau,\tau]$ when $i\ne j$
except when $i=3$ and $j=4$ (recall that in (\ref{eq:defank}) we
left a little space in the sides of the intervals --- this is the
reason). This immediately shows $|h_{j}(x)|=1$. Further, $h_{j}'=h_{j}(\eta_{j}'+i\widetilde{\eta}_{j}')$
gives $|h_{j}'|=|\widetilde{\eta}_{j}'|$ and $h_{j}''=h_{j}((\eta_{j}'+i\widetilde{\eta}_{j}')^{2}+\eta_{j}''+i\widetilde{\eta}_{j}'')$
gives $|h_{j}''|\leq|\widetilde{\eta}_{j}'|^{2}+|\widetilde{\eta}_{j}''|$.
Now, the derivatives of $\widetilde{\eta}_{j}$ have the representations
\[
\widetilde{\eta}_{j}'(x)=\int_{\mathbb{T}}\eta_{j}(x-t)H'(t)\, dt\quad\widetilde{\eta}_{j}''(x)=\int_{\mathbb{T}}\eta_{j}(x-t)H''(t)\, dt\]
where $H$ is the Hilbert kernel. This again works because $\eta_{j}$
is zero in a neighborhood of $x$, otherwise there would be extra
terms (the derivatives of $H$ are distributions with a singular part
at $0$, but here we may consider them as just functions). We may
therefore estimate\begin{equation}
\widetilde{\eta}_{j}'(x)\leq\left\Vert \eta\right\Vert _{2}\left\Vert H'|_{\mathbb{T}\setminus[-\tau_{r},\tau_{r}]}\right\Vert _{2}\qquad\widetilde{\eta}_{j}''(x)\leq\left\Vert \eta\right\Vert _{2}\left\Vert H''|_{\mathbb{T}\setminus[-\tau_{r},\tau_{r}]}\right\Vert _{2}.\label{eq:utilde}\end{equation}
Now, it is well known that for any $D\ge0$, the $D^{\textrm{th}}$
derivative of $H$ satisfies \begin{equation}
|H^{(D)}(t)|\leq\frac{C(D)}{|e^{it}-1|^{D+1}}.\label{eq:Hilbert}\end{equation}
So \[
\left\Vert H'|_{[-\tau_{r},\tau_{r}]^{c}}\right\Vert _{2}\approx\tau_{r}^{-3/2}\qquad\left\Vert H''|_{[-\tau_{r},\tau_{r}]^{c}}\right\Vert _{2}\approx\tau_{r}^{-5/2}.\]
and since $\left\Vert \eta\right\Vert _{2}\leq Cn$ we get the estimate
we need:\[
|h_{j}'|\leq C\tau_{r}^{-3/2}n\qquad|h_{j}''|\leq C\tau_{r}^{-3}n^{2}.\]
We may now fix the value of $C_{1}$ from the definition of $\beta$.
With this the conditions of lemma \ref{lem:real} are fulfilled and
we are done.
\end{proof}

\begin{proof}
[End of the proof of theorem \ref{thm:thickness}]We return to the
bound on \[
X:=\int_{K_{n}}f_{n}(t)e^{-imt}\, dt.\]
We shall do so by estimating $\mathbb{E}|X|^{4}$. Let\[
E(k_{1},k_{2},k_{3},k_{4}):=\left|\mathbb{E}\prod\mathcal{I}_{k_{i}}\right|\]
let $r(k_{1},\dotsc,k_{4})$ be the minimal $r$ such that the $I(n,k_{i})$-s
are contained in at least $3$ distinct intervals of rank $r$. A
simple calculation shows \[
\#\{(k_{1},\dotsc,k_{4}):r(k_{1},\dotsc,k_{4})=r\}\approx2^{4n-2r}.\]

If $\tau_{r}$ is too small then the estimate of lemma \ref{lem:complex}
is useless and it would be better to estimate $|E(k_{1},\dotsc,k_{4})|\leq\alpha^{4}$.
Let $R$ be some number. Then for large $r$ we have the estimate\begin{equation}
E_{1}:=\sum_{r(k_{1},\dotsc,k_{4})\geq R}E(k_{1},\dotsc,k_{4})\leq C\alpha^{4}2^{4n-2R}\stackrel{(\ref{eq:defalpha})}{\le}Ce^{4\delta n}2^{-2R}\le Cm^{C\delta}2^{-2R}.\label{eq:E1}\end{equation}

For small $r$ we use the lemma to get a better estimate. Examine
one such $k_{1},\dotsc,k_{4}$ and let $r=r(k_{1},\dotsc,k_{4})$.
Lemma \ref{lem:complex} gives\[
E(k_{1},\dotsc,k_{4})\leq\alpha^{4}\frac{Cn^{2}}{m^{2}\tau_{r}^{3}}=\frac{\alpha^{4}}{m^{2-o(1)}\tau_{r}^{3}}\quad.\]
Therefore\begin{align}
E_{2}: & =\sum_{r(k_{1},\dotsc,k_{4})<R}E(k_{1},\dotsc,k_{4})\leq\alpha^{4}2^{4n}m^{-2+o(1)}\sum_{r=1}^{R}2^{-2r}\tau_{r}^{-3}\stackrel{(\ref{eq:deftaun},\ref{eq:defalpha})}{=}\label{eq:E2}\\
 & =m^{-2+C\delta+o(1)}\sum_{r=1}^{R}2^{r+o(r)}=m^{-2+C\delta+o(1)}2^{R+o(R)}.\nonumber \end{align}
Taking $R=\left\lfloor \frac{2}{3}\log_{2}m\right\rfloor $ and summing
(\ref{eq:E1}) and (\ref{eq:E2}) we get\begin{equation}
\mathbb{E}|X|^{4}\leq m^{-4/3+C\delta+o(1)}.\label{eq:EX4m43}\end{equation}
With $\delta$ sufficiently small this is summable. Hence \[
\mathbb{E}\sum_{m}\left|X_{m}\right|^{4}=\sum\mathbb{E}\left|X_{m}\right|^{4}<\infty\]
In particular, with probability $1$, $X_{m}\to0$. As remarked above,
this shows that $\widehat{f_{n}}(m)\rightarrow0$ and hence $\widehat{F}(m)\to0$
which concludes the theorem.
\end{proof}
Let us remark that we do not know if the probabilistic part of the
proof is really necessary. In fact it is quite possible that the non-probabilistic
construction (i.e.\ setting all $s(n,k)$ to 0) also satisfies that
$\widehat{F}(m)\to0$. What we do know is that the theorem does not
follow formally from the corollary, i.e.\ there exists a singular
distribution $S$ with $\widehat{S}(n)\in l^{2}(\mathbb{Z}^{-})$
which is not a pseudo-function. The construction is very similar to
that of the theorem except one takes $\delta$ --- the $\delta$ from
the definition $F=\exp(\delta(G+i\widetilde{G}))$ --- large rather
than small. It then follows that $\left\Vert F_{n}\right\Vert _{2}$
is a large power of $m$ leading to $\widehat{F}(m)$ being unbounded.
We omit all other details.

\section*{III}

Let now $K$ be some fixed compact. Suppose $K$ supports a non-trivial
measure with some one-sided smoothness property. Does it imply that
$K$ in fact supports a measure with a two-sided property? What about
distributions? We summarize some relations of this sort in the tables
in figure \ref{fig:K}.%
\begin{figure}
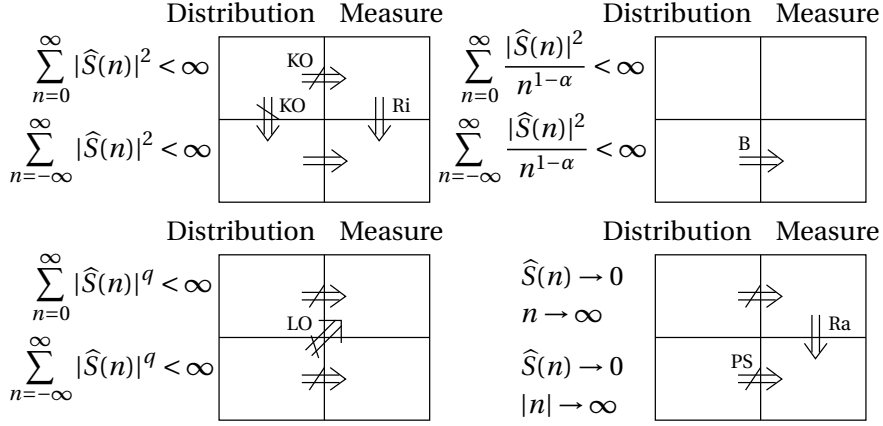

\caption{\label{fig:K}Properties for specific compacts.}

\end{figure}
 Each arrow denotes an implication which holds for any compact. For
example, the arrow labelled {}``B'' in the top right corner is a
result of Beurling that any compact $K$ which supports a distribution
$S$ with $\sum|\widehat{S}(n)|^{2}|n|^{\alpha-1}<\infty$ also supports
a measure $\mu$ with the same property \cite{B49}. In the top-right
table we assume $\alpha<1$ and in the bottom-left $2<q<\infty$.
The abbreviations are KO for \cite{KO.03}, Ri for the Riesz analyticity
theorem, PS for Piatetski-Shapiro \cite{P54} and Ra for Rajchman
\cite{R29}. The arrow marked LO comes from \cite{LO.05} and noting
that the set constructed in \cite{LO.05} is a Helson set, so it cannot
support even a measure $\mu$ with $\widehat{\mu}(n)\to0$. The unmarked
arrow on the top-left is trivial, while the unmarked arrow on the
bottom-right and the two unmarked arrows on the bottom left follow
from the other arrows in their diagrams. All missing arrows are unknown
to us.

Our last result is to show that a certain {}``non-symmetry'' is
possible for singular measures supported by compacts of dimension
$\alpha<1$. It does not fit in any of the tables above, but it is
very close in spirit. 
\begin{thm}
\label{thm:asym}Let $K$ be a compact on $\mathbb{T}$ , $\dim_{\F}(K)=d>0$.
Then for every $p>2/d$ there exists a (complex) measure $\nu\in\mathcal{M}(K)$
such that\[
\widehat{\nu}\in l^{p}(\mathbb{Z}_{-}),\qquad\widehat{\nu}\notin l^{p}(\mathbb{Z}_{+}).\]

\end{thm}
(recall the definition of the Fourier dimension on page \pageref{def:Fdim}).
\begin{lem}
\label{lem:ol1}Let $\mu$ be a non-atomic positive measure, $g$
be a function in $C(\mathbb{T})$. Then for every $\delta>0$ there
is a positive integer $l$ such that \begin{equation}
\int\left|g(lt)\right|d\mu<\int|g(t)|d\mu+\delta\label{eq:lem9}\end{equation}
\end{lem}
\begin{proof}
Clearly it is enough to consider $g$ as a positive trigonometric
polynomial. Assume first that $\widehat{\mu}(n)=o(1)$ as $|n|\to\infty$
(the only case needed for the proof of Theorem \ref{thm:asym}). Then
we have \[
\int g(lt)d\mu=\sum\widehat{g}(n)\overline{\widehat{\mu}(ln)}\]
and for a large $l$ the right-hand side is $<\widehat{g}(0)\widehat{\mu}(0)+\delta$
which gives (\ref{eq:lem9}).

For a general non-atomic $\mu$ we can use Wiener's theorem \cite[\S I.7.11]{K76}
to get \[
\frac{1}{2N+1}\sum_{|n|\le N}|\widehat{\mu}(n)|^{2}=o(1).\]
From this it follows that for a random $l_{N}$ between $N/2K$ and
$N/K$, \[
\lim_{N\to\infty}\sum_{0<|n|\le K}|\widehat{\mu}(l_{N}n)|=0\]
and the argument can be completed in the same way.
\end{proof}

\begin{proof}
[Proof of Theorem \ref{thm:asym}] Let $\dim_{\F}=d$ , $p>2/d$.
According to the definition of the Fourier dimension we can find a
positive measure $\mu\in\mathcal{M}(K)$ such that $\widehat{\mu}\in l^{p}(\mathbb{Z})$.
We can assume that $||\widehat{\mu}||_{l^{p}(\mathbb{Z})}=1$. Now
we define inductively a sequence of (complex) measures $\nu_{k}$.
Let us describe the $k^{\textrm{th}}$ step of the induction. Choose
a number $s=s_{k}>0$ such that $||\widehat{\mu}||_{l^{p}(-\infty,-s)}<2^{-k}$.
Set \[
g_{k}(t):=4^{-k}\sum_{j=1}^{4^{k}}e^{iq(j)t}\]
By choosing the frequencies $s<q(1)<...<q(4^{k})$ sparse enough
one can satisfy the inequality: \[
||\widehat{\mu g_{k}}||_{l^{p}(\mathbb{Z}^{+})}>1/2||\widehat{\mu}||_{l^{p}}||\widehat{g_{k}}||_{l^{1}}=1/2\]
and it will hold true when replace $g_{k}(t)$ by $g_{k}(lt)$ for
$l>1$. So for $\nu_{k}:=g_{k}(l)\mu$ we will have $||\nu_{k}||_{l^{p}(I_{k})}>1/2$
for a certain finite interval $I_{k}\subset\mathbb{Z}_{+}$, and $||\nu_{k}||_{l^{p}(\mathbb{Z}_{-})}<2^{-k}$.
In addition, choosing $l$ from lemma \ref{lem:ol1} for sufficiently
small $\delta$ we get from (\ref{eq:lem9}) \[
\left\Vert \nu_{k}\right\Vert _{\mathcal{M}(\mathbb{T})}\stackrel{\eqref{eq:lem9}}{<}2\left\Vert g_{k}\right\Vert _{1}\left\Vert \mu\right\Vert _{\mathcal{M}(\mathbb{T})}\le2\left\Vert g_{k}\right\Vert _{2}=2^{-k+1}\]
Certainly in the above induction we can get the intervals $I_{k}$
to be disjoint. It follows that the measure $\nu=\sum\nu_{k}$ satisfies
the requirements of the theorem.\end{proof}
\begin{rem*}
The restriction $p>2/d$ in theorem \ref{thm:asym} is essentially
sharp. Indeed, take a Salem set $K$ of given dimension, that is $\dim K=\dim_{F}K=d$.
Suppose there is a measure $\nu\in\mathcal{M}(K)$ such that $\widehat{\nu}\in l^{p}(\mathbb{Z}_{-})$
for some $p<2/d$. Then the condition (\ref{eq:one-sided}) is satisfied
for some $\alpha>d$. So Theorem \ref{thm:main} implies $\dim K>d$.
Which is a contradiction.
\end{rem*}

\subsection*{Questions}
\begin{enumerate}
\item In the spirit of Theorem \ref{thm:asym} one may ask the following.
Let $K$ be a compact with $\dim K=d$. Does it follow that for any
$0<\alpha<d$ there is a measure $\mu$ supported on $K$ with \begin{equation}
\sum_{n<0}\frac{|\widehat{\mu}(n)|^{2}}{|n|^{1-\alpha}}<\infty\qquad\sum_{n>0}\frac{|\widehat{\mu}(n)|^{2}}{|n|^{1-\alpha}}=\infty?\label{eq:asym l2w}\end{equation}
The following strict version is also open: Let $K$ be some compact
which supports a measure $\mu$ with $\sum_{n\neq0}|\widehat{\mu}(n)|^{2}|n|^{\alpha-1}<\infty$.
Is it always true that it also supports a measure $\mu$ with (\ref{eq:asym l2w})?
\item Does there exist a compact of uniqueness $K$ (see e.g.\ the book
\cite{KL87} for sets of uniqueness) which supports a pseudo-measure
$S$ with $\widehat{S}(n)=o(1)$ as $n\to+\infty$? One can show that
Cantor sets with any ratio of dissection do not (in fact any $H^{(n)}$-set
does not support such a pseudo-measure, the argument of Piatetski-Shapiro
applies, see \cite[chap. III, thm. 4]{KL87}).
\item One can introduce $\dim_{\lp}(K)$ as \[
\sup\{a:K\mbox{ supports a measure }\mu\mbox{ such that }\widehat{\mu}\in l^{2/a}(\mathbb{Z})\}\,.\]
Then $\dim_{\F}K\le\dim_{\lp}K\le\dim K$. Is it any easier to construct
a set with $\dim K=\dim_{\lp}K$ than a Salem set? One may call such
sets {}``quasi-Salem sets''.
\end{enumerate}

\subsection*{Acknowledgements}

Work partially supported by the authors' respective Israel Science
Foundation grants.

\end{document}